\documentclass{amsart}      
\usepackage{amssymb,amsmath,graphicx}
\newtoks\prt 
 
\newtheorem{thm}{Theorem}[section]
 
\newtheorem{lemma}[thm]{Lemma} 
\newtheorem{prop}[thm]{Proposition} 
 
\newtheorem{example}[thm]{Example}

\newtheorem*{question}{Question} 
\theoremstyle{definition}

\def\eqn#1$$#2$${\begin{equation}\label#1#2\end{equation}}

\def\M{\mathcal M}

\def\ce{\mathbb C}

\def\en{\mathbb N} 
\def\er{\mathbb R} 
 
\def\ef{\mathbb F}

\def \reg {\partial _{\kern1pt\text{reg}}}

\def\dd{\operatorname{d}}
\def\di{\,\mbox{\rm d}}
\def\dh{\widehat{\operatorname{d}}}
\def\clu#1#2{\operatorname{clust}_{#1^{**}}(#2)}

\newcommand{\ca}[1]{\operatorname{ca}\left(#1\right)}

\newcommand{\cc}{\operatorname{cc}}
\newcommand{\wk}[2][X]{\operatorname{wk}_{#1}\left(#2\right)}

\newcommand{\wscl}[1]{\overline{#1}^{w^*}}
\newcommand{\upleq}{\rotatebox{90}{$\,\leq\ $}}

\let\dd\dist

\begin{document}

\title{Quantification of the reciprocal Dunford-Pettis property}
\author{Ond\v{r}ej F.K. Kalenda and Ji\v{r}\'{\i} Spurn\'y}

\address{Department of Mathematical Analysis \\
Faculty of Mathematics and Physic\\ Charles University\\
Sokolovsk\'{a} 83, 186 \ 75\\Praha 8, Czech Republic}
\email{kalenda@karlin.mff.cuni.cz}
\email{spurny@karlin.mff.cuni.cz}

%\address{Universit\'e d'Orl\'eans\\ 
%BP 6759\\
%F-45067 Orl\'eans Cedex 2\\France}
%\email{hermann.pfitzner@univ-orleans.fr}

%\address{Department of Mathematical Analysis \\
%Faculty of Mathematics and Physic\\ Charles University\\
%Sokolovsk\'{a} 83, 186 \ 75\\Praha 8, Czech Republic}


\thanks{Our research was supported in part by the grant
GA\v{C}R P201/12/0290.}

\begin{abstract} 
We prove in particular that Banach spaces of the form $C_0(\Omega)$,
where $\Omega$ is a locally compact space, enjoy a quantitative 
version of the reciprocal Dunford-Pettis property. 
\end{abstract}
\maketitle

%%%%%%%%%%%%%%%%%%%%%%%%%%%%%%%%%%%%%%%%%%%%%%%%%%%%

\section{Introduction}

A Banach space $X$ is said to have the \emph{Dunford-Pettis property} if, for any Banach space $Y$ every 
weakly compact operator $T:X\to Y$ is completely continous.
Further, $X$ is said to have the \emph{reciprocal Dunford-Pettis property} if, for any Banach space $Y$ every  completely continous
operator $T:X\to Y$ is weakly compact.

Let us recall that $T$ is \emph{weakly compact} if the image by $T$ of the unit ball of $X$ is relatively weakly compact in $Y$. Further, $T$ is \emph{completely continous} if it maps weakly convergent sequences to norm convergent ones, or, equivalently, if it maps weakly Cauchy sequence to norm Cauchy (hence norm convergent) ones.

In general, these two classes of operators are incomparable.
For example, the identity on $\ell_2$ is weakly compact (due to reflexivity of $\ell_2$) but not completely continuous. On the other hand, the identity on $\ell_1$ is completely continuous
(by the Schur property) and not weakly compact.

It is obvious that reflexive spaces have the reciprocal Dunford-Pettis property (as any operator with reflexive domain is weakly compact) and that the Schur property implies the Dunford-Pettis property (as any operator defined on a space with the Schur property is completely continuous). Moreover, the space $L^1(\mu)$ has the Dunford-Pettis property for any non-negative $\sigma$-additive measure $\mu$ by (see  \cite[Theorem 1]{gro} or \cite[p. 61--62]{JoLi}). The space $C_0(\Omega)$, where $\Omega$ is a locally compact space, has both the Dunford-Pettis property
and the reciprocal Dunford-Pettis property (see \cite[p. 153, Theorem 4]{gro}).

In the present paper we investigate quantitative versions of the reciprocal Dun\-ford-Pettis property. It is a kind of a continuation of a recent paper \cite{qdpp} where quantification of the Dunford-Pettis property is studied.
It is also related to many  results on quantitative versions of certain theorems and properties. In particular,
quantitative versions of Krein's theorem were studied in \cite{f-krein,Gr-krein,GHM,CMR}, quantitative versions of Eberlein-\v{S}mulyan and Gantmacher theorems were investigated in \cite{AC-meas}, a quantitative version of James' compactness theorem was proved in \cite{CKS,Gr-james}, quantification of weak sequential continuity and of the Schur property was addressed in \cite{wesecom,qschur}.

The main idea behind quantitative versions is an attempt to replace the respective implication by an inequality. So, in case of the reciprocal Dunford-Pettis property we will try to replace the implication 
$$T\mbox{ is completely continuous}\Rightarrow T\mbox{ is weakly compact} $$
by an inequality of the form
\begin{multline*} \mbox{measure of weak non-compactness of }T \\ \leq C\cdot \mbox{measure of non-complete-continuity of }T.\end{multline*}

We will use the same quantities as in \cite{qdpp} and in addition some equivalent ones. 

In \cite{qdpp} it is proved, in particular, that both $L^1(\mu)$ spaces and $C_0(\Omega)$ spaces enjoy the strongest possible version of quantitative Dunford-Pettis property. In the present paper, we show that $C_0(\Omega)$ spaces have also a quantitative version of the reciprocal Dunford-Pettis property.

\section{Preliminaries}

In this section we define the quantities used in the present paper
and recall some known relationships between them. Most of the quantities we investigate are taken from \cite{qdpp} but we will
need a few more.

We will need to measure how far a given operator is from being
weakly compact, completely continuous or Mackey compact. 

Our results are true both for real and complex spaces. In the real case sometimes better constants are obtained. By $\ef$ we will denote $\er$ or $\ce$, depending on whether we consider real or complex spaces.

\subsection{Measuring non-compactness of sets}
In this subsection we define measures on non-compactness, weak non-compactness and Mackey non-compactness of sets. We start by recalling the Hausdorff measure of non-compactness in metric spaces
and one of its equivalents.

Let $(X,\rho)$ be a metric space. If $A,B\subset X$ are two
nonempty sets, their non-symmetrized Hausdorff distance is defined
by
$$\dh(A,B)=\sup\{\dd(x,B):x\in A\}.$$
The Hausdorff measure of non-compactness of a nonempty set $A\subset
X$ is defined by
$$\chi(A)=\inf\left\{\dh(A,F):F\subset X\mbox{ finite}\right\}.$$
Then $\chi(A)=0$ if and only if $A$ is totally bounded. In case $(X,\rho)$ is complete this is equivalent to relative compactness of $A$. We will need the following ``absolute'' equivalent:
$$\chi_0(A)=\inf\left\{\dh(A,F):F\subset A\mbox{ finite}\right\}.$$
The quantity $\chi_0(A)$ depends only on the metric stucture of $A$ itself, not on the space $X$ where it is embedded. It is easy to check that
\begin{equation}
\chi(A)\le\chi_0(A)\le 2\chi(A)
\label{eq:chi-chi0}
\end{equation}
for any nonempty set $A\subset X$.

If $X$ is a Banach space and $A\subset X$ a nonempty bounded set,
we define the following two measures of weak noncompactness of $A$:
$$\begin{aligned}
\omega(A)&=\inf\left\{\dh(A,K): K\subset X\mbox{ weakly compact}\right\},\\
\wk{A}&=\dh\left(\wscl{A},X\right).
\end{aligned}$$
The quantity $\omega(A)$ is the de Blasi measure of weak noncompactness introduced in \cite{deblasi} and later investigated
for example in \cite{tylli-cambridge,AC-meas,qdpp}. The quantity $\wk{A}$ was used (with various notations) for example in \cite{f-krein,f-subswcg,Gr-krein,GHM,Gr-james,AC-meas,AC-jmaa,CMR,CKS}. These two quantities are not equivalent (see \cite{tylli-cambridge,AC-meas,qdpp}), while there
are several other natural quantities equivalent to the second one (see the papers quoted above). In general we have the following inequalities
\begin{equation}
\label{eq:kompaktnost} \wk{A}\le\omega(A)\le \chi(A)
\end{equation}
for any nonempty bounded subset $A\subset X$. These inequalities are easy,
the first one is proved for example in \cite{AC-meas}. The second one is obvious as finite sets are weakly compact.

Although the quantities $\omega(A)$ and $\wk{A}$ are not equivalent in general,
in some spaces they are equal. In particular, by \cite[Proposition 10.2 and Theorem 7.5]{qdpp} we have
\begin{equation}\begin{aligned}
\label{eq:c0L1}X=c_0(\Gamma)\mbox{ or }X=&L^1(\mu)\mbox{ for a non-negative $\sigma$-additive measure }\mu \\ &\Rightarrow \wk{A}=\omega(A)\mbox{ whenever }A\subset X\mbox{ is bounded.}
\end{aligned}
\end{equation}

We continue by measuring Mackey non-compactness. Let $X$ be still a Banach space.
Suppose that $A\subset X^*$ is a nonempty bounded set. Let us recall that the Mackey topology on $X^*$ is the topology of uniform convergence on weakly compact subsets of $X$. Moreover, the Mackey topology is complete, hence relatively compact subsets
coincide with totally bounded ones. So, $A$ is relatively Mackey compact if and only 
if 
$$A|_L=\{x^*|_L:x^*\in A\}$$
is totally bounded in $\ell_\infty(L)$ for each weakly compact set $L\subset X$.
This inspires the following definition:
$$\chi_m(A)=\sup\{\chi_0(A|_L): L\subset B_X\mbox{ weakly compact}\}.$$
This quantity measures Mackey non-compactness in the sense that $\chi_m(A)=0$ if and only if $A$ is relatively Mackey compact.

\subsection{Measuring of non-compactness of operators}
Let $T:X\to Y$ be a bounded operator. Since $T$ is compact (weakly compact)
if and only if $TB_X$ is relatively compact (relatively weakly compact) in $Y$,
it is natural to measure (weak) noncompactness of $T$ by a quantity applied
to $TB_X$. To simplify the notation we set
$$\chi(T)=\chi(T B_X),\chi_0(T)=\chi_0(T B_X),\omega(T)=\omega(T B_X),
\wk[Y]{T}=\wk[Y]{TB_X}.$$
Similarly, if $Y=Z^*$ for a Banach space, we set
$$\chi_m(T)=\chi_m(TB_X).$$

There are quantitative versions of Schauder's and Gantmacher's theorems
on compactness and weak compactness of the dual operators. More precisely,
if $T:X\to Y$ is a bounded linear operator, we have
\begin{gather}\label{eq:q-Schauder}
\frac12\chi(T^*)\le\chi(T)\le 2\chi(T^*),\qquad \frac12\chi_0(T^*)\le\chi_0(T)\le 2\chi_0(T^*),\\
\label{eq:Gant}\frac12\wk[X^*]{T^*}\le\wk[Y]{T}\le2\wk[X^*]{T^*},\\
\label{eq:noneq}\mbox{the quantities $\omega(T)$ and $\omega(T^*)$ are incomparable in general.}
\end{gather}

The first part of \eqref{eq:q-Schauder} follows from \cite{GM},
the second part follows for example from Lemma~\ref{lm-TA} below applied to the identity operator on $X$ and $A=B_X$. The assertion  \eqref{eq:Gant} follows from \cite[Theorem 3.1]{AC-meas}. The last assertion is proved in \cite[Theorem 4]{tylli-cambridge}.

\subsection{Measuring non-complete-continuity}
In this subsection we introduce two quantities which measure how far
an operator is from being completely continuous. The first one is that used in \cite{qdpp} and it is based on the definition of complete continuity given in the introduction.  Let us start by defining a quantity measuring how far 
a given sequence is from being norm-Cauchy.

Let $(x_k)$ be a bounded sequence in a Banach space. Following
\cite{qschur,qdpp} we set
$$\ca{x_k}=\inf_{n\in\en}\sup\{\|x_k-x_l\|: k,l\ge n\}.$$
It is clear that $\ca{x_k}=0$ if and only if $(x_k)$ is norm-Cauchy.

Further, let $T:X\to Y$ be a bounded operator between Banach spaces.
Following \cite{qdpp} we set
$$\cc(T)=\sup\{\ca{Tx_k}:(x_k)\mbox{ is a weakly Cauchy sequence in }B_X\}.$$
It is clear that $T$ is completely continuous if and only if $\cc(T)=0$.

We will need one more equivalent quantity. This is inspired by an equivalent description of completely continuous operators. The operator $T$ is completely
continuous if and only if $T(L)$ is norm-compact for any weakly compact set $L\subset X$. Such operators are sometimes called Dunford-Pettis, so we will use the following
notation.
$$\cc_{DP}(T)=\sup\{\chi_0(TL):L\subset B_X\mbox{ weakly compact} \}.$$

The two quantities are equivalent. More precisely, we have
\begin{equation}
\label{eq:ccdp} \cc_{DP}(T)\le\cc(T)\le 2\cc_{DP}(T)
\end{equation}

Let us provide a proof. Suppose that $\cc_{DP}(T)>c>0$. Fix a weakly compact 
set $L\subset B_X$ with $\chi_0(TL)>c$. It is easy to construct by induction
a sequence $(y_k)$ in $TL$ with $\|y_k-y_l\|>c$ for any $1\le l<k$.
Let $x_k\in L$ be such that $Tx_k=y_k$. By weak compactness of $L$ we can without
loss of generality suppose that $(x_k)$ is weakly convergent and hence weakly Cauchy.
Since $\ca{Tx_k}\ge c$, we get $\cc(T)\ge c$. This completes the proof of the first inequality.

To show the second one, suppose $\cc(T)>c>0$. Let $(x_k)$ be a weakly Cauchy sequence in $B_X$ with
$\ca{Tx_k}>c$. We can find two sequences $(m_k)$ and $(n_k)$ of natural numbers
such that for each $k\in\en$ we have $m_k<n_k<m_{k+1}$ and $\|Tx_{n_k}-Tx_{m_k}\|>c$. Set $y_k=\frac12(x_{n_k}-x_{m_k})$.
Then $(y_k)$ is a weakly null sequence in $B_X$ and hence $L=\{y_k:k\in\en\}\cup\{0\}$ is a weakly compact subset of $B_X$.

We claim that $\chi_0(TL)\ge\frac c2$. Suppose not. Then there is a finite set
$F\subset TL$ and 
%a number $d$ 
with $\dh(TL,F)<\frac c2$. Since $F$ is finite, there is $h\in F$ and  a subsequence $(y_{k_l})$ such that $\|T y_{k_l}-h\|<\frac c2$ for each $l\in\en$.
Since $(T y_{k_l}-h)$ weakly converges to $-h$, we get $\|h\|\le \frac c2$, so
$h=0$. (Any other element of $TL$ has norm strictly greater than $\frac c2$.)

So,  $\|T y_{k_l}\|<\frac c2$ for each $l\in\en$. But this is a contradiction with the choice of the sequence $(y_k)$.

Therefore $\chi_0(TL)\ge\frac c2$ and so $\cc_{DP}(T)\ge\frac c2$.
This completes the proof of the second inequality.

\section{Main results}

Our first main result is the following theorem which says that
spaces $C_0(\Omega)$ enjoy a quantitative version of the reciprocal Dunford-Pettis property. The formulation combines this result with 
a result of \cite{qdpp} on the quantitative Dunford-Pettis property. We thus get that for operators defined on a $C_0(\Omega)$ space
the weak compactness and complete continuity are quantitatively equivalent.

\begin{thm}\label{main-t} Let $X=C_0(\Omega)$, where $\Omega$ is a Hausdorff locally compact space. Let $Y$ be any Banach space and $T:X\to Y$ be a bounded linear operator. Then we have
$$\frac1{4\pi}\wk[Y]{T}\le\cc(T)\le 4\wk[Y]{T}.$$
More precisely, the following inequalities hold:
\begin{equation}
\label{eq:main1} 
\begin{aligned}
\frac1{4\pi}\wk[Y]{T}\le\frac1{2\pi}&\wk[X^*]{T^*}= \frac1{2\pi}\omega(T^*)
\le\cc_{DP}(T)\\&\le \cc(T)\le 2\omega(T^*) = 2\wk[X^*]{T^*}\le 4\wk[Y]{T}.
\end{aligned}
\end{equation}
In case of real-valued functions the constant $\pi$ in the above inequalities can be everywhere replaced by the constant $2$.
\end{thm}

This theorem says, in particular, that the quantities $\cc(T)$, $\wk[Y]{T}$, $\wk[X^*]{T^*}$ and $\omega(T^*)$ are equivalent
for any bounded linear operator $T:C_0(\Omega)\to Y$. 

The first inequality, as well as the last one follows from \eqref{eq:Gant}. The two equalities follow from \eqref{eq:c0L1}
as $C_0(\Omega)^*$ is of the form $L^1(\mu)$. The inequality $\cc(T)\le 2\omega(T^*)$ follows from \cite[Theorem 5.2]{qdpp} as
$C_0(\Omega)$ has the Dunford-Pettis property. The inequality
$\cc_{DP}(T)\le\cc(T)$ follows from \eqref{eq:ccdp}.

Finally, the main new result is the inequality 
$\frac1{2\pi}\omega(T^*)\le\cc_{DP}(T)$ which follows from 
Theorems~\ref{cc-Mc} and~\ref{t:M=w} below.

\medskip 

It is natural to ask whether also the quantity $\omega(T)$ is equivalent
to the remaining ones. The answer is positive in case $\Omega$ is scattered. Indeed, then $X^*$ is isometric to $\ell_1(\Omega)$ and hence we can use \cite[Theorem 8.2]{qdpp}.

In general the answer is negative as witnessed by the following example.

\begin{example}\label{main-ex} There is a separable Banach space $Y$ such that for any uncountable separable metrizable locally compact space $\Omega$ there is
a sequence $(T_n)$ of bounded operators $T_n:C_0(\Omega)\to Y$ such that
$$\lim_{n\to\infty}\frac{\wk[Y]{T_n}}{\omega(T_n)}
=\lim_{n\to\infty}\frac{\cc(T_n)}{\omega(T_n)}=0.$$
\end{example}

This example is proved in Section~\ref{sec-exa} below.

\section{Complete continuity and Mackey compactness}

In this section we prove a quantitative version of a particular case
of \cite[Lemma 2]{gro}. The quoted lemma implies that an operator
$T:X\to Y$ is completely continuous if and only if its adjoint $T^*$
is Mackey-compact (i.e., $T^*(B_{Y^*})$ is relatively Mackey compact in $X^*$). A quantitative version is the following theorem.

\begin{thm}\label{cc-Mc}
Let $X$ and $Y$ be Banach spaces and $T:X\to Y$ an operator.
Then
$$\frac12\chi_m(T^*)\le\cc_{DP}(T)\le 2\chi_m(T^*)$$
\end{thm}

The proof of this theorem is done by a refinement of the arguments
in \cite{gro}. The first tool used in \cite{gro} is the  Arzel\`a-Ascoli theorem. We will use its quantitative version for a special case of $1$-Lipschitz functions on a metric space. It is contained in the following lemma.

\begin{lemma}\label{lm-ascoli} Let $(M,\rho)$ be a metric space and $A\subset \ell_\infty(M)$ be a bounded set formed by $1$-Lipschitz functions.
Then $\chi_0(A)\le 2\chi_0(M)$.
\end{lemma}

\begin{proof}
Fix an arbitrary $c>\chi_0(M)$ and $\varepsilon>0$. Then there is a finite set $F\subset M$ such that $\dh(M,F)<c$. Let us define the mapping
$\Phi:A\to\ef^F$ by $\Phi(f)=f|_F$ for $f\in A$. Let us equip $\ef^F$ with the $\ell_\infty$ norm. Then $\Phi(A)$ is 
a bounded subset of $\ef^F$, so it is also totally bounded (as $F$ 
is finite). It follows that there is a finite set $B\subset A$ such that $\dh(\Phi(A),\Phi(B))<\varepsilon$.

We will show that $\dh(A,B)\le 2c+ \varepsilon$. To this end take an arbitrary $f\in A$. By the choice of $B$ there is some $g\in B$ with $\|\Phi(f)-\Phi(g)\|< \varepsilon$. Fix an arbitrary $x\in M$.
We can find $x_0\in F$ with $\rho(x_0,x)<c$. Then
$$\begin{aligned}
|f(x)-g(x)|&\le |f(x)-f(x_0)|+|f(x_0)-g(x_0)|+|g(x_0)-g(x)|
\\&\le \rho(x,x_0)+ \|\Phi(f)-\Phi(g)\| + \rho(x_0,x)
<2c+ \varepsilon.
\end{aligned}$$

Hence $\|f-g\|<2c+ \varepsilon$, so $\dd(f,B)<2c+ \varepsilon$. Since $f\in A$ is arbitrary, we get
$\dh(A,B)\le2c+ \varepsilon$, in particular $\chi_0(A)\le 2c+ \varepsilon$. Finally, since $c>\chi_0(M)$ and $\varepsilon>0$ are arbitrary, we get $\chi_0(A)\le 2\chi_0(M)$. 
\end{proof}

The next lemma is a quantitative version of a part of \cite[Lemma 3]{gro}. It is formulated in a very abstract setting. 

\begin{lemma}\label{twosets}
Let $A$ be a nonempty set, $B$ a nonempty bounded subset of $\ell_\infty(A)$. Let $\varphi:A\to\ell_\infty(B)$ be 
defined by
$$\varphi(a)(b)=b(a),\qquad a\in A, b\in B.$$
Then
$$\frac12\chi_0(B)\le\chi_0(\varphi(A))\le 2\chi_0(B).$$
\end{lemma}

\begin{proof} It is clear that $\varphi(A)$ is a bounded subset of
$\ell_\infty(B)$. Moreover, it is formed by $1$-Lipschitz functions.
Indeed, let $a\in A$ be arbitrary. Fix $b_1,b_2\in B$. Then
$$|\varphi(a)(b_1)-\varphi(a)(b_2)|=|b_1(a)-b_2(a)|\le\|b_1-b_2\|.$$
So, by Lemma~\ref{lm-ascoli} we have
$$\chi_0(\varphi(A))\le 2\chi_0(B).$$
To show the second inequality, let us define a canonical
embedding $\psi:B\to\ell_\infty(\varphi(A))$ by
$$\psi(b)(\varphi(a))=b(a),\qquad b\in B, a\in A.$$
This is a well-defined mapping. Indeed, if $a_1,a_2\in A$ are such that $\varphi(a_1)=\varphi(a_2)$, then for any $b\in B$
we have
$$b(a_1)=\varphi(a_1)(b)=\varphi(a_2)(b)=b(a_2).$$
Moreover, $\psi$ is an isometry of $B$ onto $\psi(B)$. Indeed, let $b_1,b_2\in B$ be arbitrary. Then
$$\begin{aligned}\|\psi(b_1)-\psi(b_2)\|& = \sup \{|\psi(b_1)(x)-\psi(b_2)(x)|: x\in \varphi(A)\} 
\\&= \sup \{|\psi(b_1)(\varphi(a))-\psi(b_2)(\varphi(a))|: a\in A\} 
\\&= \sup \{|b_1(a)-b_2(a)|: a\in A\}  = \|b_1-b_2\|.\end{aligned}$$
It follows that $\psi(B)$ is isometric to $B$. Moreover, $\psi(B)$ is a bounded subset of $\ell_\infty(\varphi(A))$ made from $1$-Lipschitz functions (the argument is the same as the one used above in the proof of the first inequality). Hence, using Lemma~\ref{lm-ascoli} we get
$$\chi_0(B)=\chi_0(\psi(B))\le 2\chi_0(\varphi(A)).$$
This completes the proof.
\end{proof} 

The next lemma is a quantitative version of \cite[Lemma 2]{gro}
applied to a single set rather than to a family of sets. 

\begin{lemma}\label{lm-TA} Let $T:X\to Y$ be an operator between Banach spaces.
Let $A\subset X$ be a bounded set.  Let $\psi: X^*\to\ell_\infty(A)$ denote the restriction mapping.
Then
$$\frac12\chi_0(TA)\le \chi_0(\psi(T^*B_{Y^*}))\le 2\chi_0(TA).$$
\end{lemma}

\begin{proof}
Let us denote by $\kappa$ the canonical embedding of $Y$ into $\ell_\infty(B_{Y^*})$. Let us define an embedding $\varphi:B_{Y^*}\to \ell_\infty(\kappa(TA))$ by 
$$\varphi(y^*)(\kappa(y))=y^*(y),\qquad y\in TA, y^*\in B_{Y^*}.$$
Using Lemma~\ref{twosets} and the fact that $\kappa$ is an isometry,
we obtain
$$\frac12\chi_0(TA)\le \chi_0(\varphi(B_{Y^*}))\le 2\chi_0(TA).$$
Further, $\psi(T^*B_{Y^*})$ is isometric to $\varphi(B_{Y^*})$.
Indeed, the mapping
$\alpha:\varphi(B_{Y^*})\to\psi(T^*B_{Y^*})$
defined by
$$\alpha(\varphi(y^*))=\psi(T^*y^*),\qquad y^*\in B_{Y^*},$$
is an onto isometry. Let $y_1^*,y_2^*\in B_{Y^*}$ be arbitrary.
Then
$$\begin{aligned}\|\psi(T^*y_1^*)-\psi(T^*y_2^*)\|
&=\sup\{|(T^*y_1^*)(a)-(T^*y_2^*)(a)|:a\in A\}
\\&=\sup\{|y_1^*(Ta)-y_2^*(Ta)|:a\in A\}
=\|\varphi(y_1^*)-\varphi(y_2^*)\|.\end{aligned}$$
It follows that $\alpha$ is a well-defined isometry. Moreover,
it is clear that it is surjective.
Hence we get $\chi_0(\psi(T^*B_{Y^*}))=\chi_0(\varphi(B_{Y^*}))$.
So,
$$\frac12\chi_0(TA)\le \chi_0(\psi(T^*B_{Y^*}))\le 2\chi_0(TA).$$
and the proof is completed.
\end{proof}

Now we are ready to proof the theorem.

\begin{proof}[Proof of Theorem~\ref{cc-Mc}]
The inequalities follow from Lemma~\ref{lm-TA} by taking supremum over all weakly compact sets $A\subset B_X$.
\end{proof}

\section{Weak compactness and Mackey compactness in spaces of measures}

Within this section $\Omega$ will denote a locally compact space
and $\M(\Omega)$ will be the space of all finite Radon measures on $\Omega$ equipped with the total variation norm and considered as the dual space to $C_0(\Omega)$. We will consider simultaneuously the real version (i.e., $C_0(\Omega)$ are real-valued functions and $\M(\Omega)$ are signed measures) and complex version (i.e., $C_0(\Omega)$ are complex functions and $\M(\Omega)$ are complex measures) of these spaces.

We will prove a quantitative version of a result of \cite{gro} saying that in $\M(\Omega)$ weakly compact sets coincide with Mackey compact ones. In \cite{gro} this result is hidden in Corollary to Theorem 2 on page 149 and in the first two lines on page 150.
The promised quantitative version is the following theorem. 

\begin{thm}\label{t:M=w}
Let $A\subset \M(\Omega)$ be a bounded set.
Then
$$\frac12\chi_m(A)\le \omega(A)\le \pi\chi_m(A).$$
In case of real measures, the constant $\pi$ can be replaced by $2$.
\end{thm}

The first step to the proof is a quantiative version of a modification of \cite[Theorem 2]{gro}. In the quoted theorem several equivalent conditions for weak compactness of a subset of $\M(\Omega)$ are summarized. We will prove  quantitative versions of some of them and of some others. They are contained in 
the following proposition.

Let us comment this result a bit.
The second quantity is inspired by condition (2) of \cite[Theorem 2]{gro}. The first inequality follows directly from \cite{gro}.
The third quantity is inspired by condition (3) of the quoted theorem. The second inequality is easy and is done by copying the
respective proof from \cite{gro}.

It is also easy to quantify the implication (3)$\Rightarrow$(4)
from \cite{gro}, but we were not able to quantify the last implication saying that the condition (4) implies weak compactness.
Instead, we used the fourth quantity. The proof of the third inequality required a new idea. The last inequality is proved using
technics from \cite{qdpp}.

\begin{prop}\label{q-dual}
Let $A$ be a bounded subset of $\M(\Omega)$. Then
$$\begin{array}{c}
\omega(A)\\
\upleq
\\
\sup\left\{ \limsup\limits_{k\to\infty} \sup\limits_{\mu\in A} \left|\int f_k\di\mu\right| : (f_k)\mbox{ is a weakly null sequence in }B_{C_0(\Omega)}\right\}
\\ \upleq \\
\sup\left\{ \limsup\limits_{k\to\infty} \sup\limits_{\mu\in A} |\mu(U_k)| :\begin{aligned}(U_k)&\mbox{ is a sequence of}\\ &\mbox{ pairwise disjoint}\mbox{ open subsets of }\Omega\end{aligned}\right\}\\
\upleq
\\
\sup\left\{ \limsup\limits_{k\to\infty} \sup\limits_{\mu\in A} |\mu(F_k)| : \begin{aligned}(F_k)&\mbox{ is a sequence of}\\ &\mbox{ pairwise disjoint}\mbox{ compact subsets of }\Omega\end{aligned}\right\}
\\ \upleq \\
\frac1\pi\omega(A).
\end{array}$$
In the real case the constant $\frac1\pi$ can be replaced by $\frac12$.
\end{prop}

\begin{proof} The first two inequalites easily follow from \cite{gro}. 
Indeed, let $(f_k)$ be any weakly null sequence in $B_{C_0(\Omega)}$. Fix an arbitrary $c>\omega(A)$. Then there
is a weakly compact set $H\subset \M(\Omega)$ such that $A\subset H+c B$, where $B$ denotes the unit ball of $\M(\Omega)$.  By \cite[Theorem 2]{gro}
$$\lim \int f_k\di\mu = 0\qquad \mbox{uniformly for }\mu\in H.$$
Thus
$$\begin{aligned}
\limsup\limits_{k\to\infty} \sup\limits_{\mu\in A} \left|\int f_k\di\mu\right|
&\le\limsup\limits_{k\to\infty} \sup\limits_{\mu\in H+cB} \left|\int f_k\di\mu\right|
\\&\le
\limsup\limits_{k\to\infty}\left( \sup\limits_{\mu\in H} \left|\int f_k\di\mu\right|+c
\sup\limits_{\mu\in B} \left|\int f_k\di\mu\right|\right)
\\&\le \limsup\limits_{k\to\infty}\left( \sup\limits_{\mu\in H} \left|\int f_k\di\mu\right|+c\right)= c.\end{aligned}$$
Since $c>\omega(A)$ is arbitrary, we get
$$\limsup\limits_{k\to\infty} \sup\limits_{\mu\in A} \left|\int f_k\di\mu\right|\le\omega(A),$$
so the first inequality follows.

Let us show the second inequality. If the third quantity is zero, the inequality is obvious. So, suppose that the quantity is strictly positive and fix an arbitrary smaller positive constant $c$.
Then there is a sequence $(U_k)$ of pairwise disjoint open sets in $\Omega$ and a sequence $(\mu_k)$ in $A$ such that $|\mu_k(U_k)|>c$ for each $k\in\en$. For each $k\in\en$ we can find a continuous function $f_k:\Omega\to[0,1]$ supported by a compact subset of $U_k$ such that $\left|\int f_k\di\mu_k\right|>c$. Since $(f_k)$ is a bounded sequence in $C_0(\Omega)$ which pointwise converges to zero, it is weakly null. This completes the proof of the second inequality.

Let us proceed with the third inequality. Obviously, it is enough
to prove the following lemma.

\begin{lemma}\label{compact-open}
Let $(\mu_k)$ be a bounded sequence in $\M(\Omega)$, let $(F_k)$ be a sequence of pairwise disjoint compact subsets of $\Omega$. Let $c>0$ be such that $|\mu_k(F_k)|>c$ for each $k\in\en$. Then for any $\varepsilon>0$
there is a subsequence $(\mu_{k_n})$ and a sequence $(U_n)$ of  pairwise disjoint open subsets such that $|\mu_{k_n}(U_n)|>c- \varepsilon$ for each $n\in\en$:
\end{lemma}

\begin{proof} Let $(\mu_k)$, $(F_k)$ and $c$ satisfy the assumptions. Let $\varepsilon>0$ be arbitrary. Set
$$\gamma=\sup\{\|\mu_k\|:k\in\en\}$$
and find $N\in\en$ with $\frac1N<\frac{\varepsilon}{\gamma}$.
We will construct by induction for each $n\in\en$
a natural number $k_n$, sets of natural numbers $M_n$ 
and an open set $U_n\subset \Omega$ such that the following
conditions are satisfied for each $n\in\en$:

\begin{itemize}
	\item $k_n\in M_n$,
	\item $M_{n+1}\subset M_n\setminus\{k_n\}$, $M_{n+1}$ is infinite,
	\item $|\mu_k|(F_k\cap \overline{U_n})<\frac{\varepsilon}{2^n}$ for each $k\in M_{n+1}$,
	\item $\overline{U_n}\cap\overline{U_j}=\emptyset$ for $j<n$,
	\item $|\mu_{k_n}(U_n)|>c- \varepsilon$.
\end{itemize}

We start by setting $M_1=\en$.

Suppose that $n\in\en$ and that we have already constructed $M_j$ for $j\le n$ and $k_j$, $U_j$ for $j<n$. Since $M_n$ is infinite,
we can fix a subset $H\subset M_n$ of cardinality $2^n N$. We can find  open sets $(V_h)_{h\in H}$ with pairwise disjoint closures such that $F_h\subset V_h$ for each $h\in H$.

For any $k\in M_n\setminus H$ there is some $h(k)\in H$ such that
$|\mu_k|\left(F_k\cap \overline{V_{h(k)}}\right)<\frac{\varepsilon}{2^n}$. 
Fix $k_n\in H$ such that
$$M_{n+1}=\{k\in M_n\setminus H: h(k)=k_{n} \}$$
is infinite.

By the induction hypothesis we have $|\mu_{k_n}|(F_{k_n}\cap\overline{U_j})<\frac{\varepsilon}{2^j}$ for any $j<n$. It follows that $|\mu_{k_n}(F_{k_n}\setminus\bigcup_{j<n} \overline{U_j})|>c- \varepsilon$. It follows that there is a compact set $L\subset F_{k_n}\setminus\bigcup_{j<n} \overline{U_j}$ with $|\mu_{k_n}(L)|>c- \varepsilon$.
Finally, we can find an open set $U_n$ such that
\begin{itemize} 	
	\item $L\subset U_n\subset V_{k_n}$,
	\item $\overline{U_n}\cap\overline{U_j}=\emptyset$ for $j<n$,
	\item $|\mu_{k_n}(U_n)|>c- \varepsilon$.	
\end{itemize}
This completes the induction step and the lemma is proved.
\end{proof}

Now we come back to the proof of Proposition~\ref{q-dual}. 
It remains to prove the last inequality. If $\omega(A)=0$, it is trivial. Suppose that $\omega(A)>0$ and fix an arbitray $c\in(0,\omega(A))$. We will proceed in three steps. 

In the first step we reduce the problem to a statement on
$L^1$ spaces on a finite measure space. Since $\M(\Omega)$ is an $L^1$ space on a $\sigma$-additive non-negative measure (infinite, of course), by \cite[Theorem 7.5]{qdpp} there is a sequence $(\mu_k)$ in $A$ such that $\dd(\clu{\M(\Omega)}{\mu_k})>c$.
Set $\mu=\sum_{k=1}^\infty 2^{-k}|\mu_k|$. Then $\mu$ is a finite measure and $L^1(\mu)$ is canonically isometrically embedded into $\M(\Omega)$ onto a subspace containing the sequence $(\mu_k)$.
Set $\tilde A=\{\mu_k:k\in \en\}$.
By \cite[Proposition 7.1]{qdpp}, the quantity $\omega(\tilde A)$ is the same in $\M(\Omega)$ as in the subspace identified with $L^1(\mu)$. In particular $\omega(\tilde A)>c$ in $L^1(\mu)$.

The second step will be the following lemma.

\begin{lemma} Let $\mu$ be a finite $\sigma$-additive nonnegative measure and $\tilde A\subset L^1(\mu)$ be a bounded subset satisfying $\omega(\tilde A)>c>0$. Then for any $\varepsilon>0$ there is a sequence $(f_k)$ in $\tilde A$ and a sequence $(H_k)$ of pairwise disjoint measurable sets satisfying $\int_{H_k}|f_k|>c- \varepsilon$ for each $k\in\en$.
\end{lemma}

\begin{proof}
We will use the construction from the proof of \cite[Proposition 7.1]{qdpp}. Let $B$ denote the unit ball of $L^\infty(\mu)$. Then $B$ is a weakly compact subset of $L^1(\mu)$ and hence $\dh(\tilde A,\alpha B)\ge\omega(\tilde A)>c$ for each $\alpha>0$.
Set $\gamma=\sup\{\|f\|:f\in\tilde A\}$.
We will construct by induction positive numbers $\alpha_k$ and functions $f_k\in\tilde A$ such that
\begin{itemize}
	\item $\dd(f_k,\alpha_k B)>c$,
	\item $\alpha_{k+1}>\alpha_k$,
	\item $\int_E |f_j|\di\mu<\frac{\varepsilon}{2^k}$ whenever $j\le k$ and $\mu(E)\le\frac\gamma{\alpha_{k+1}}$.
\end{itemize}

Set $\alpha_1=1$. Having $\alpha_k$, we can find $f_k\in\tilde A$ satisfying the first condition. Further, by absolute continuity we can find $\alpha_{k+1}>\alpha_k$ such that the third condition is satisfied. This completes the construction.

Set $E_k=\{t:|f_k(t)|>\alpha_k\}$. Then
$$\int_{E_k} |f_k|\di\mu=\dd(f_k,\alpha_k B)>c.$$
Further,
$\mu(E_k)\le\frac{\|f_k\|}{\alpha_k}\le\frac{\gamma}{\alpha_k}$.
It follows that for any $j<k$ we have
$$\int_{E_k}|f_j|\di\mu<\frac{\varepsilon}{2^{k-1}}.$$

Finally, set $H_k=E_k\setminus\bigcup_{n>k}E_n$. Then $(H_k)$ is a sequence of pairwise disjoint measurable sets and
$$\int_{H_k} |f_k|\di\mu
\ge \int_{E_k} |f_k|\di\mu - \sum_{n>k}\int_{E_n} |f_k|\di\mu
>c- \varepsilon.$$
This completes the proof.
\end{proof}

Now we return to the proof of Proposition~\ref{q-dual}.
Using the first step and the above lemma we obtained, given $\varepsilon>0$, a sequence $(\mu_k)$ in $A$ (a subsequence of the sequence chosen in the first step) and a sequence $(H_k)$ of pairwise disjoint sets which are measurable for each $\mu_k$ such that $|\mu|(H_k)>c-\varepsilon$.

As the third step we will use the following lemma.

\begin{lemma} Let $\mu$ be a finite complex measure. Let $H$ be a measurable set with $|\mu|(H)>d>0$. Then there is a measurable subset $\tilde H\subset H$ with $|\mu(\tilde H)|>\frac d\pi$.
 
If $\mu$ is real-valued, $\tilde H$ can be found to satisfy  $|\mu(\tilde H)|>\frac d2$.
\end{lemma}

\begin{proof} The real-valued case is easy using the Hahn decomposition $\mu=\mu^+-\mu^-$.

Let us prove the general case. Since $|\mu|(H)>c$, there are
pairwise disjoint measurable sets $D_1,\dots,D_p\subset H$ such that
$\sum_{j=1}^p |\mu(D_j)|>c$. By \cite[Lemma 6.3]{rudin} there is a subset $J\subset\{1,\dots,p\}$ such that $\left|\sum_{j\in J} \mu(D_j)\right|>\frac c\pi$. It is enough to take $\tilde H=\bigcup_{j\in J}D_j$.
\end{proof}

Now we are ready to finish the proof. Find $\mu_k$-measurable sets $\tilde H_k\subset H_k$ with $|\mu_k(\tilde H_k)|>\frac {c-\varepsilon}\pi$
and then a compact subset $F_k\subset \tilde H_k$ with 
$|\mu_k(F_k)|>\frac {c-\varepsilon}\pi$.

In case of real measures we can obtain $|\mu_k(F_k)|>\frac {c-\varepsilon}2$. 
This completes the proof of the last inequality.    
\end{proof}

The last lemma of this section is a quantitative version 
of \cite[p. 134, Corollary to Lemma 3]{gro}.

\begin{lemma}\label{mac-seq} Let $X$ be a Banach space and $A\subset X^*$ a bounded set. Then
\begin{multline*}\frac18\chi_m(A)\le \sup\left\{\limsup_{k\to\infty}\sup_{x^*\in A}
|x^*(x_k)| : (x_k)\mbox{ is a weakly null sequence in }B_X\right\}
\\ \le\chi_m(A)\end{multline*}
\end{lemma}

\begin{proof}
Let $(x_k)$ be a weakly null sequence in $B_X$. Set 
$L=\{0\}\cup\{x_k:k\in\en\}$. Then $L$ is a weakly compact
subset of $B_X$, so $\chi_0(A|_L)\le\chi_m(A)$. Let $\varepsilon>0$
be arbitrary. Then there is a finite set $F\subset A$ such 
that $\dh(A|_L,F|_L)<\chi_m(A)+\varepsilon$. For each $k\in\en$ we have
$$\sup_{x^*\in A} |x^*(x_k)|\le \sup_{x^*\in F} |x^*(x_k)| + \chi_m(A)+\varepsilon.$$
Since
$$\lim_{k\to\infty}\sup_{x^*\in F} |x^*(x_k)|=0,$$
we get 
$$\limsup_{k\to\infty}\sup_{x^*\in A} |x^*(x_k)|\le \chi_m(A)+\varepsilon.$$
Since $\varepsilon>0$ is arbitrary, this completes the proof of the
second inequality.

Let us prove the first one. If $\chi_m(A)=0$, the inequality is obvious. So, suppose that $\chi_m(A)>c>0$. Fix a weakly compact
set $L\subset B_X$ such that $\chi_0(A|_L)>c$. 
Let $\varphi:X\to\ell_\infty(A)$ be the canonical mapping
defined by $$\varphi(x)(x^*)=x^*(x), \qquad x\in X, x^*\in A.$$
Further, let us define $\varphi_0:L\to \ell_\infty(A|_L)$ by
$$\varphi_0(x)(x^*|_L)= x^*(x) \qquad x\in L, x^*\in A.$$
It is clear that $\varphi_0$ is well defined and that $\varphi_0(L)$ is isometric to $\varphi(L)$.
By Lemma~\ref{twosets} we have
$\chi_0(\varphi_0(L))>\frac c2$, hence also $\chi_0(\varphi(L))>\frac c2$. Therefore,
we can construct by induction a sequence $(x_k)$ in $L$ such that 
$$\|\varphi(x_k)-\varphi(x_l)\|>\frac c2,\qquad 1\le l<k.$$
Since $L$ is weakly compact, we can suppose without loss of generality that the sequence $(x_k)$ weakly converges to some $x\in L$. Then 
$$\|\varphi(x_k)-\varphi(x)\|>\frac c4$$
for all $k\in\en$ with at most one exception. So, suppose without loss of generality that it holds for each $k\in\en$. Hence, if we set
$y_k=\frac12(x_k-x)$, then $(y_k)$ is a weakly null sequence in $B_X$ and
$$\sup_{x^*\in A}|x^*(y_k)|=\|\varphi(y_k)\|>\frac c8.$$
This completes the proof of the first inequality. 
\end{proof}

Finally, we are ready to prove the theorem.

\begin{proof}[Proof of Theorem~\ref{t:M=w}.]
Let $A\subset \M(\Omega)$ be a bounded set. 
It follows from Proposition~\ref{q-dual}
and Lemma~\ref{mac-seq} that
$$\frac18\chi_m(A)\le\omega(A)\le\pi\chi_m(A).$$
The second inequality is the announced one, the first one still
needs to be improved. So, suppose that $\omega(A)<c$. Fix $H\subset \M(\Omega)$ weakly compact such that $\dh(A,H)<c$. Then $H$ is a bounded subset of $\M(\Omega)$ satisfying $\omega(H)=0$, hence also
$\chi_m(H)=0$ (we apply the above inequality to $H$). 

Given $L\subset B_{C_0(\Omega)}$ weakly compact and $\varepsilon>0$, there is a finite set $F\subset H$ such 
that $\dh(H|_L,F|_L)<\varepsilon$. Then clearly
$\dh(A|_L,F|_L)<c+ \varepsilon$, so $\chi(A|_L)\le c+ \varepsilon$.
By \eqref{eq:chi-chi0} we get
$\chi_0(A|_L)\le 2(c+ \varepsilon)$. 

It follows that $\chi_m(A)\le 2c$, so $\chi_m(A)\le 2\omega(A)$. 
This completes the proof of the first inequality.
\end{proof}

\section{Proof of Example~\ref{main-ex}}\label{sec-exa}

Let $\Delta=\{-1,1\}^\en$ be the Cantor space.
Let $\Omega$ be an uncountable separable metrizable locally compact space.
Denote by $K$ its one-point compactification. Then $K$ is an uncountable metrizable compact space, therefore $C(K)$ is isomorphic to $C(\Delta)$ by Milyutin's theorem \cite{mil}
(see, e.g., \cite[Theorem 2.1]{rosenthal}). So, $C_0(\Omega)$, which is a hyperplane in $C(K)$, is isomorphic to
a hyperplane in $C(\Delta)$. Since hyperplanes in $C(\Delta)$ are isomorphic to $C(\Delta)$, we can conclude that $C_0(\Omega)$ is isomorphic to $C(\Delta)$. 

Fix an onto isomorphism $Q:C_0(\Omega)\to C(\Delta)$. Then 
$$\frac1{\|Q^{-1}\|} B_{C(\Delta)}\subset Q(B_{C_0(\Omega)})\subset\|Q\|
 B_{C(\Delta)},$$
 and hence for any bounded opearator $T:C(\Delta)
\to Y$ we have 
$$
\begin{aligned}
\cc(TQ)&\le \|Q\|\cc(T),\\
\wk[Y]{TQ}&=\wk[Y]{TQ(B_{C_0(\Omega)})}\le\wk[Y]{T(\|Q\|B_{C(\Delta)})}
=\|Q\|\wk[Y]{T},\\
\omega(TQ)&=\omega(TQ(B_{C_0(\Omega)}))\ge \omega\left(T\left(\frac1{\|Q^{-1}\|}B_{C(\Delta)}\right)\right)=\frac1{\|Q^{-1}\|}\omega(T).
\end{aligned}$$

It follows that it is enough to restrict ourselves to the case $\Omega=\Delta$.

Let $\mu$ denotes the product probability measure on $\Delta$. I.e., 
$\mu$ is the countable power of the uniform probability measure
$\frac12(\delta_1+\delta_{-1})$ on the two-point set $\{-1,1\}$.
Let us define an equivalent norm $\|\cdot\|_n$ on $C(\Delta)$ by
$$\|f\|_n = \frac1n\|f\|+\int_\Delta |f| \di\mu,\qquad f\in C(\Delta).$$
Set $Y_n=(C(\Delta),\|\cdot\|_n)$ and let 
$$Y =\left(\;\bigoplus_{n\in\en} Y_n\right)_{c_0}$$ 
be the $c_0$-sum of the spaces $Y_n$.

Let $Q_n: C(\Delta)\to Y_n$ be the identity mapping, $I_n:Y_n\to Y$ the canonical inclusion made by completing by zeros and $P_n:Y\to Y_n$ be the canonical projection. Let us define $T_n= I_nQ_n$. Then $T_n$ is an operator from $C(\Delta)$ to $Y$. 
The proof will be completed if we show that
$$\frac1{4\pi}\wk[Y]{T_n}\le\cc(T_n)\le \frac2n,\qquad \omega(T_n)\ge \frac12.$$
The first inequality follows from Theorem~\ref{main-t}. Let us show the second one. Let $(f_k)$ be a weakly Cauchy sequence in $B_{C(\Delta)}$.
Then the sequence $(f_k)$ pointwise converges to a (not necessarily continuous) function $f$. Since the sequence is uniformly bounded,
the Lebesgue dominated convergence theorem shows that
$$\lim_{k\to\infty}\int_\Delta |f_k-f|\di\mu=0.$$
In particular, $(f_k)$ is Cauchy in the norm of $L^1(\mu)$. Thus, 
given $\varepsilon>0$ there is $k_0\in\en$ such that whenever
$k,l\ge k_0$ we have $\|f_k-f_l\|_{L^1(\mu)}<\varepsilon$, hence
$$\|f_k-f_l\|_n = \frac1n \|f_k-f_l\| + \int_\Delta |f_k-f_l|\di\mu
<\frac2n+ \varepsilon.$$
It follows that for each $n\in\en$ we have
$$\ca{T_n f_k}=\ca{Q_n f_k}\le \frac2n+\varepsilon.$$
Since $\varepsilon>0$ is arbitrary and $(f_k)$ is an arbitrary weakly Cauchy sequence in $B_{C(\Delta)}$, we obtain $\cc(T_n)\le \frac2n$.

We finish by proving the third inequality. We will prove it by contradiction. Suppose that $\omega(T_n)<c<\frac12$. Let us fix a weakly compact set $L_0\subset Y$ with $\dh(T_n(B_{C(\Delta)}),L_0)<c$. Since 
$T_n(B_{C(\Delta)})\subset I_n(Y_n)$, we have
$$\dh(Q_n(B_{C(\Delta)}),P_n(L_0))\le\dh(T_n(B_{C(\Delta)}),L_0)<c.$$
Set $L=P_n(L_0)$. Then $L$ is a weakly compact subset of $Y_n$.

For any $k\in\en$ let $\pi_k:\Delta\to \{-1,1\}$ be the projection on the $k$-th coordinate. It is a continuous function from $B_{C(\Delta)}$. 
So, there is $y_k\in L$ such that $\|y_k-Q_n(\pi_k)\|_n<c$.
Since $L$ is weakly compact, there is a subsequence $(y_{k_j})$ weakly 
converging to some $y\in L$. Set $f_{k_j}=Q_n^{-1}(y_{k_j})$.
 Since $Q_n$ is an isomorphism, the sequence
$(f_{k_j})$ is weakly convergent in $C(\Delta)$. So it is uniformly bounded and pointwise convergent, hence by the Lebesgue dominated theorem it
is Cauchy in the $L^1$ norm. Let $0<\varepsilon<1-2c$. Fix $j_0$ such that
for $i,j\ge j_0$ we have
$$\int_\Delta |f_{k_i}-f_{k_j}|\di\mu <\varepsilon.$$
Fix $i>j\ge j_0$. Then 
$$\begin{aligned}1&=\int_\Delta |\pi_{k_i}-\pi_{k_j}|\di\mu
\\&\le \int_\Delta |\pi_{k_i}-f_{k_i}|\di\mu
+\int_\Delta |f_{k_i}-f_{k_j}|\di\mu
+\int_\Delta |f_{k_j}-\pi_{k_j}|\di\mu
\\&<\|Q_n(\pi_{k_i})-y_{k_i}\|+ \varepsilon+\|y_{k_j}-Q_n(\pi_{k_j})\|
<2c+ \varepsilon,\end{aligned}$$
which is a contradiction completing the proof.

\section{Final remarks and open questions}

The first natural question is the following one:

\begin{question}
Are the constants in the inequalities in Theorem~\ref{main-t} optimal?
\end{question}

Another natural problem concerns other spaces with the reciprocal Dunford-Pettis property.

\begin{question} Is there a Banach space which enjoys the reciprocal Dunford-Pettis property but not a quantitative version?
\end{question}

By a quantitative version we mean  the existence of a constant $C$ such that the inequality
$$\wk[Y]{T}\le C\cdot\cc(T)$$
holds for any operator $T:X\to Y$. 

Let us remark that for the Dunford-Pettis property there is a quantitative version which is automatically satisfied, see \cite[Theorem 5.2]{qdpp}. We do not know whether a similar thing
holds for the reciprocal Dunford-Pettis property. Our proofs strongly used the structure of $C_0(\Omega)$ spaces.

Let us explain what seems to be a difference between these two properties.

It follows from \cite[Proposition 1]{gro} that
\begin{equation}
\label{eq:DPP} 
\begin{aligned}X&\mbox{ has the Dunford-Pettis property} \\&\qquad\qquad\Leftrightarrow
\mbox{any weakly compact subset of $X^*$ is Mackey compact.}
\end{aligned}
\end{equation}
Futher, \cite[Proposition 8]{gro} implies that 
\begin{equation}
\label{eq:RDPP} 
\begin{aligned}X&\mbox{ has the reciprocal Dunford-Pettis property} \\&\qquad\qquad\Leftrightarrow
\mbox{any Mackey compact subset of $X^*$ is weakly compact.}
\end{aligned}
\end{equation}

Hence, suppose that $X$ has the Dunford-Pettis property. Then any bounded set $A\subset X^*$ satisfies $\chi_m(A)\le 2\omega(A)$,
see the final part of the proof of Theorem~\ref{t:M=w}. And this yields an automatic quantitative version of the Dunford-Pettis property.

We are not able to proceed similarly for the reciprocal Dunford-Pettis property. If $X$ has the reciprocal Dunford-Pettis
property and $A\subset X^*$ is bounded, we do not know how to control $\omega(A)$ by $\chi_m(A)$. We know that any Mackey compact
is weakly compact, thus, if we define
$$\omega_m(A)=\inf\{\dh(A,H): H\subset X^*\mbox{ Mackey compact}\},$$
we obtain $\omega(A)\le\omega_m(A)$. But it is not clear, whether
$\omega_m(A)$ can be controlled by $\chi_m(A)$. (Conversely, $\chi_m(A)\le 2\omega_m(A)$ by the final part of the proof of Theorem~\ref{t:M=w}.) This inspires the following question:

\begin{question} Is the quantity $\omega_m$ defined above equivalent
to $\chi_m$?
\end{question}

For $X=C_0(\Omega)$ it is the case by Theorem~\ref{t:M=w}. But the proof essentially used the structure of $X$. We do not know the answer for general Banach spaces.

%\bibliography{qdpp}\bibliographystyle{plain}
\def\cprime{$'$} \def\cprime{$'$}

\end{document}